\documentclass{article}

\usepackage[utf8]{inputenc}
\usepackage{hyperref}
\usepackage{mathtools}

\usepackage{graphicx}%
\usepackage{multirow}%
\usepackage{amsmath,amssymb,amsfonts}%
\usepackage{amsthm}%
\usepackage{mathrsfs}%
\usepackage[title]{appendix}%
\usepackage{cite}

\newtheorem{theorem}{Theorem}

\newtheorem{lemma}[theorem]{Lemma}%

\newtheorem{remark}{Remark}%

\newtheorem{corollary}{Corollary}%

\newtheorem{assumption}{Assumption}

\date{}
\author{Alexey D. Kislovskiy\qquad\qquad Eduard Yu. Lerner\\  Igor A. Senkevich}
\title{Slack-Pack algorithm for Meir-Moser packing problem}
\begin{document}
\maketitle
\begin{abstract}
The well-known problem stated by A.~Meir and L.~Moser consists in tiling the unit square with rectangles (details), whose side lengths equal $1/n\times 1/(n+1)$, where indices~$n$ range from 1 to infinity. Recently, Terence Tao has proved that it is possible to tile with $1/n^t\times1/(n+1)^t$ rectangles (squares with the side length of $1/n^t$), $1/2<t<1$, the square, whose area equals the sum of areas of these details, provided that only those details, whose indices exceed certain~$n_0$, are taken into consideration. We adduce arguments in favor of the assumption that the result obtained by T.~Tao is also valid for $t=1$. We use a new tiling method (the Slack-Pack algorithm), which initially admits gaps between stacks of details. The algorithm uses a pre-fixed parameter $\gamma$, $\sqrt{3/2}<\gamma<3/2$, connected with the gap value. The new algorithm allows one to control the ratio of the area of the large rectangular part, which is free of details, to the whole area of the remaining empty space. This ratio (under certain natural assumptions) always exceeds $1-1/\gamma-\delta$, where $\delta$ tends to zero as~$n_0$ increases.
\end{abstract}


\section{Introduction}\label{sec1}

Consider $\frac{1}{n}\times \frac{1}{n+1}$ rectangles, $n\in\mathbb N$ (in what follows, we treat them as details). Evidently, the total area of all these rectangles equals one. We say that a packing of details in a rectangular sheet is perfect, if each detail in it intersects no other one, while the total area of all details equals the area of the sheet. In 1968, A.~Meir and L.~Moser~\cite{moser} stated the question of whether there exists a perfect packing of such details in the unit square.
They also have stated a similar square packing problem; namely, they were interested of whether there exists a perfect packing of squares, whose side lengths equal $\frac{1}{n}$, $n\in\mathbb N$, $n\geq 2$, in a rectangle, whose area equals $\sum_{n=2}^{\infty} \frac{1}{n^{2}} = \frac{\pi^{2}}{6}-1$. Though there are many papers devoted to these problems~\cite{moser, tao, paulhas, pseudopaulhas, polish, china, stack-pack, Chalcraft, joosMathRep, Janus23, Januslast}, they still remain unsolved. 

Note that problems on quadrating a rectangle are well known in combinatorics~\cite{tutte}. One of their traditional interpretations consists in constructing an electrical circuit, which transmits the unit electric current between two terminals with the potential difference of $\pi^2/6-1$ with the help of infinitely many wires, whose conductivities equal $1/2, 1/3, 1/4, \ldots$. However, in this case, as distinct from the Tutte problem, the way in which we can apply Kirchhoff laws is not clear. 

There are several algorithms for solving the Meir-Moser problem. M.~Paulhus (see \cite{paulhas}) has applied one variant of the greedy algorithm, which implies that for packing a current detail, one chooses an empty rectangular box with the minimal admissible width (if the number of such boxes is more than one, then one chooses the box with the minimal height).
Then one places the detail in a corner of the chosen box and performs the first cut. This cut (performed along the detail edge parallel to the lesser box side) results in dividing the box onto two parts (see the second illustration in Fig.~\ref{fig:1}). As a result of the second cut performed in one of these parts, one finally gets the cut-out detail and two new boxes. These boxes can be further used for packing 
details\footnote{Clearly, if the detail and the initial box have the same sizes, then no new boxes appear.}.

\begin{figure}[h]
    \centering
    \includegraphics[width=0.5\textwidth]{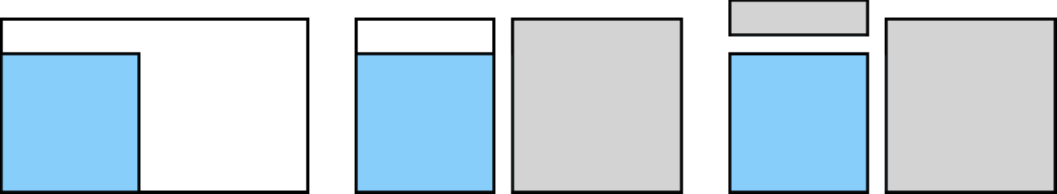}
    \caption{Illustration of the work of the Paulhus algorithm in packing details in a chosen box.}
    \label{fig:1}
\end{figure}

M.~M.~Paulhus has succeeded in packing $10^9$ details in the unit square and estimating (admitting some inaccuracies) the area of the empty space, which is necessary for packing the rest details. This result, together with the algorithm, was described more accurately in papers~\cite{pseudopaulhas} and~\cite{polish} and improved in~\cite{china}. The packing proposed in papers~\cite{paulhas} and~\cite{china} has one specific feature, namely, in one corner of the initial sheet there exists an empty rectangular space (in what follows, we call this empty space the {\it Large Rectangular Piece}, LRP), whose area constitutes a significant part of the total area of all the rest details (see~Fig.~\ref{fig:2}).
\begin{figure}[h]
    \centering
    \includegraphics[width=0.5\textwidth]{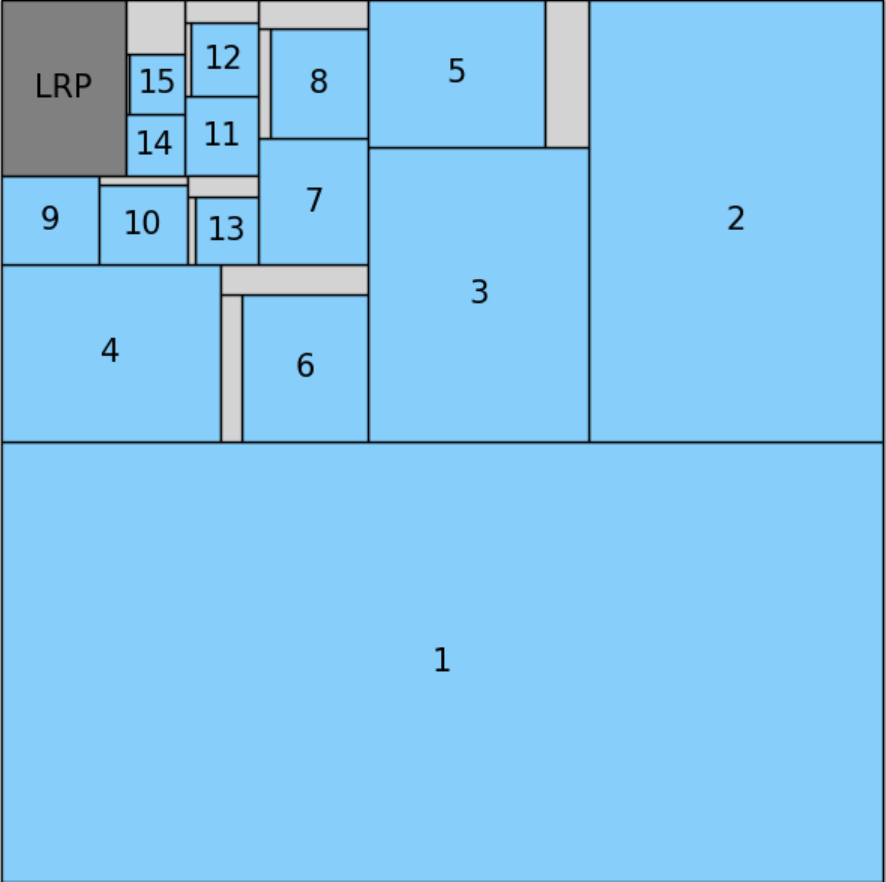}
    \caption{Illustration of the work of the Paulhus algorithm in packing first 15 rectangles with side lengths of $\frac{1}{n}\times \frac{1}{n+1}$, $n\in\mathbb N$.}
    \label{fig:2}
\end{figure}
If this space remains empty during the whole algorithm operation process, then it is possible to pack all details. However, according to results obtained in papers~\cite{paulhas} and~\cite{china}, one cannot be certain of whether the ratio of the area of the empty space in a corner of the sheet to the total area of all the rest details, which have to be packed, tends to a constant value. In experiments, where the number of packed details varied from $10^4$ to $10^{11}$, this ratio fluctuated. Namely, first it increased from $0.344$ to $0.358$, then decreased to $0.34$, then again increased to $0.37$ and then again decreased to $0.357$ (see~Table~1 in~\cite{china}).

In~\cite{stack-pack}, J.~W\"astlund considers the idea of stacking, i.e., packing of the maximal quantity of consecutively numbered details in one stripe (see~Fig.~\ref{fig:3}). This approach allows one to slightly simplify the analysis of the packing. Though it implies that stacks fit snugly together, due to different sizes of details there still exist gaps of different sizes. In paper~\cite{stack-pack}, J.~W\"astlund proves that for any $t$, $t\in (1/2,2/3)$, with sufficiently large values of the initial number $n_0$, one can perfectly pack all details that are squares with the side length of $1/n^t$, $n=n_0,n_0+1,\ldots$, in the corresponding square sheet (whose area equals the total area of all details). This result is based on the correlation between the total area and the half-perimeter of all boxes, where details have to be packed, evaluated earlier in~\cite{Chalcraft}. Namely, if the half-perimeter is less than the ratio of the total area to the size of the currently packed detail, then it is possible to pack this detail. 
\begin{figure}[h]
    \centering
    \includegraphics[width=0.5\textwidth]{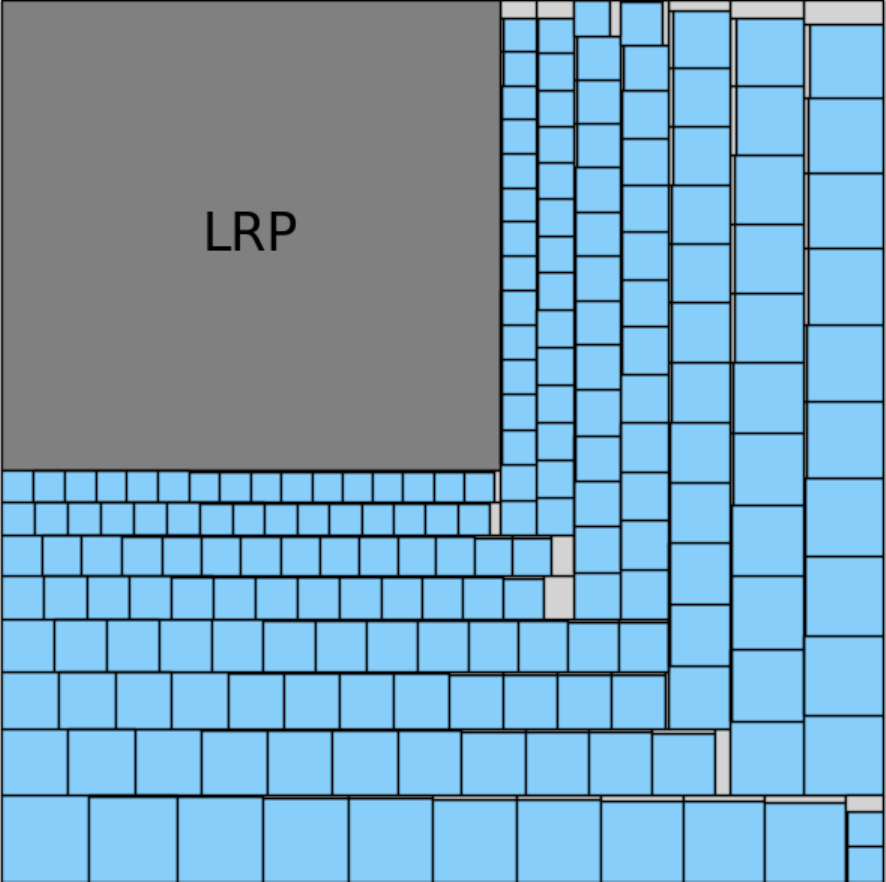}
    \caption{Illustration of the work of the W\"astlund algorithm in packing first 195 $\frac{1}{n}\times \frac{1}{n+1}$ rectangles, $n\in\mathbb N$, $n_0=100$.}
    \label{fig:3}
\end{figure}

Note that authors of the paper~\cite{Chalcraft} and subsequent papers~\cite{joosMathRep},~\cite{Janus23}, and~\cite{Januslast} studied the perfect packing problem for the square sheet, whose area equals $\zeta(2t)$, and square details, whose side length equals~$1/n^t$, where $n$ ranges from 1 to infinity. 
In particular, the existence of such packing was proved by A.~Chalcraft~\cite{Chalcraft} for $t\in [0.5864,0.6]$ and by A.~Jo\'os~\cite{joosMathRep} for $t\in [\log_3 2,2/3]$, $\ \log_3 2\approx 0.63$; J.~Januszewski and L.~Zielonka~\cite{Janus23} have extended this result for $t\in (1/2,2/3)$ and thus obtained a result stronger than that described in~\cite{stack-pack}. Finally, recently J.~Januszewski~\cite{Januslast} has proposed a very simple specific  algorithm for proving the existence of a perfect packing for $t\in(1/2,17/32]$.

In what follows, we study only the perfect packing problem for details with (almost) harmonically decreasing side lengths, starting with some $n_0^{-t}$, $n_0\in \mathbb N$. Terence Tao~\cite{tao} has made a significant progress in solving this variant of the problem. Tao has extended the result obtained by W\"astlund for the case, when the value of the parameter~$t$ is arbitrarily close~1. He has proved that $\bigl(\frac1n\bigr)^t$ squares (as well as $\bigl(\frac1n\bigr)^t\times\bigl(\frac1{n+1}\bigr)^t$) rectangles, $1/2<t<1$, can be perfectly packed in the square, whose area equals the sum of areas of all details, provided that only details, whose indices exceed certain~$n_0$, are taken into consideration. The most important property among established ones consists in the fact that the correlation between total perimeters of packed details and non-packed ones, as well as that of the perimeter and the area of non-packed details, which was used in previous papers, is also valid with all~$t<1$ (but not with $t=1$). Formally, Tao proves the obtained result by induction rather than by an explicit recursive algorithm. However, in the proof of propositions, which are the base for the induction method, he uses properties of the stack packing, where details fit snugly together.

Let us adduce arguments in favor of the assumption that the Tao result is also valid with $t=1$. In this paper, we apply a new packing algorithm, which initially admits the existence of controllable gaps between stacks of details. The gap sizes depend on a certain parameter~$\gamma$, $\sqrt{3/2}<\gamma<3/2$. Analogously to the algorithm proposed by M.M.~Paulhus, having packed a large number of details in a corner of a square sheet, one gets an empty rectangular domain. However, as distinct from the Paulhus algorithm, we can ``control'' the ratio of the area of this domain to the total area of all the rest details. Further we prove one conditional result, namely, the fact that with sufficiently large~$n_0$ this ratio exceeds $1-1/\gamma-\delta$ for arbitrarily small~$\delta$. In other words, while in other algorithms the proportion of LRP is nonstable and its asymptotic behavior is indefinite, in the Slack-Pack algorithm it stabilizes, for example, near the value of $1/4$ with $\gamma=4/3$.

The conditional character of the obtained mathematical result consists in the fact that it is based on certain assumptions (see more details in the next section), which were not proved by us. They follow, in particular, from the hypothesis (which was neither proved by us but seems to be quite natural~\cite{uniform}) that box sizes used in the algorithm are pseudo-random uniformly distributed values. Nevertheless, we have experimentally verified the result on the asymptotics of $1-1/\gamma$, it has proved to be true for packing of up to $10^{10}$ details (see~Fig.~\ref{pic:1} and~Fig.~\ref{pic:2}). Note that experimental results represent an important part of the paper content. They allow one to treat the adduced conclusions not as abstract reasonings based on certain assumptions, but as sound predictions of the relative area of the empty space in a sheet, which were obtained analytically and then confirmed experimentally.

The further part of the paper has the following structure. In Section~\ref{sec2}, we introduce the main terms, describe a packing algorithm, and state the main results. In Section~\ref{sec3}, we prove lemmas on the unification of sizes of boxes obtained in the stacking process and on their relative area. Section~\ref{sec4} is devoted to lemmas on the relative area of the rest boxes and on the final time moment of the first stage of the algorithm. In Section~\ref{sec5}, we prove the main result of this paper. In Conclusion, we describe results of numerical experiments. In Appendix~\ref{secA1}, we prove one auxiliary result on uniformly distributed random values; we need it for justifying our assumption.

\section{The description of the Slack-Pack algorithm and statement of the main results}\label{sec2}
In this section, we describe an improved algorithm for packing squares and rectangles with harmonically decreasing side lengths; we call it the Slack-Pack algorithm. At the end part of this section, we state the main result of this paper.

\subsection{Main terms and denotations}
Recall that we consider an algorithm for packing both rectangles and squares.

\textit{Details} are rectangles or squares which have to be packed; their  side lengths are assumed to decrease harmonically. 
We use the symbol $R_n$ for the $\frac{1}{n}\times \frac{1}{n+1}$ {\it rectangle}, $n\in\mathbb N$, and do $S_n$ for the $\frac{1}{n}\times \frac{1}{n}$ {\it square}, $n\in\mathbb N$. In a general case, we use the denotation $D_n$ for the {\it detail}, whose greater side equals $\frac{1}{n}$.

A \textit{sheet} is a square domain designated for packing details. The area of the sheet equals the total area of all details that have to be packed in it.

A \textit{perfect packing} is a packing such that each detail is packed in the sheet without intersections with other details, while the total area of all details equals the area of the sheet.

The \textit{large rectangular piece, (LRP)} is the largest in area rectangular piece of the sheet obtained in one of its corners as a result of the implementation of the packing algorithm. Initially, the LRP coincides with the whole sheet.

\textit{Boxes} are all the rest empty rectangular parts of the sheet, except the LRP. Initially, the set of boxes is empty.

\textit{Endpoints} are boxes obtained after the first cut made for packing a detail (see~ Fig.~\ref{fig:1}).

\textit{Normal boxes} are boxes obtained as a result of the second cut made for packing a detail (see Fig.~\ref{fig:1}, see also~Fig.~\ref{fig:4}). We use the symbol $B_n$ for the normal box obtained when packing the detail $D_n$.

The \textit{active box} is the box, where the currently considered detail has to be packed. 

Applying the stack packing method, we get boxes of the same types as those obtained by the application of the Paulhus algorithm (see~Fig.~\ref{fig:1}). A specific feature of the stacking technique consists in the fact that the new active box represents the endpoint obtained by cutting the previous active box, provided that its sizes are sufficient for packing the current detail. Therefore, each subsequent detail is being packed in the just obtained endpoint, while the latter has enough space for it (see~Fig.~\ref{fig:4}). We understand a \textit{row} as the set of details that starts with the first one that was packed in the newly chosen active box and ends with the detail such that the space that remains after its packing in the obtained endpoint is insufficient for packing one more detail. We denote the index of the {\it first detail in the current row} by the symbol $N$. 
Correspondingly, $D_N$ is the detail itself.

A \textit{stripe} is the box obtained by immediately cutting the LRP, when it is impossible to choose an active box among available boxes. A stripe always becomes an active box.

\begin{figure}[h]
    \centering
    \includegraphics[width=0.35\textwidth]{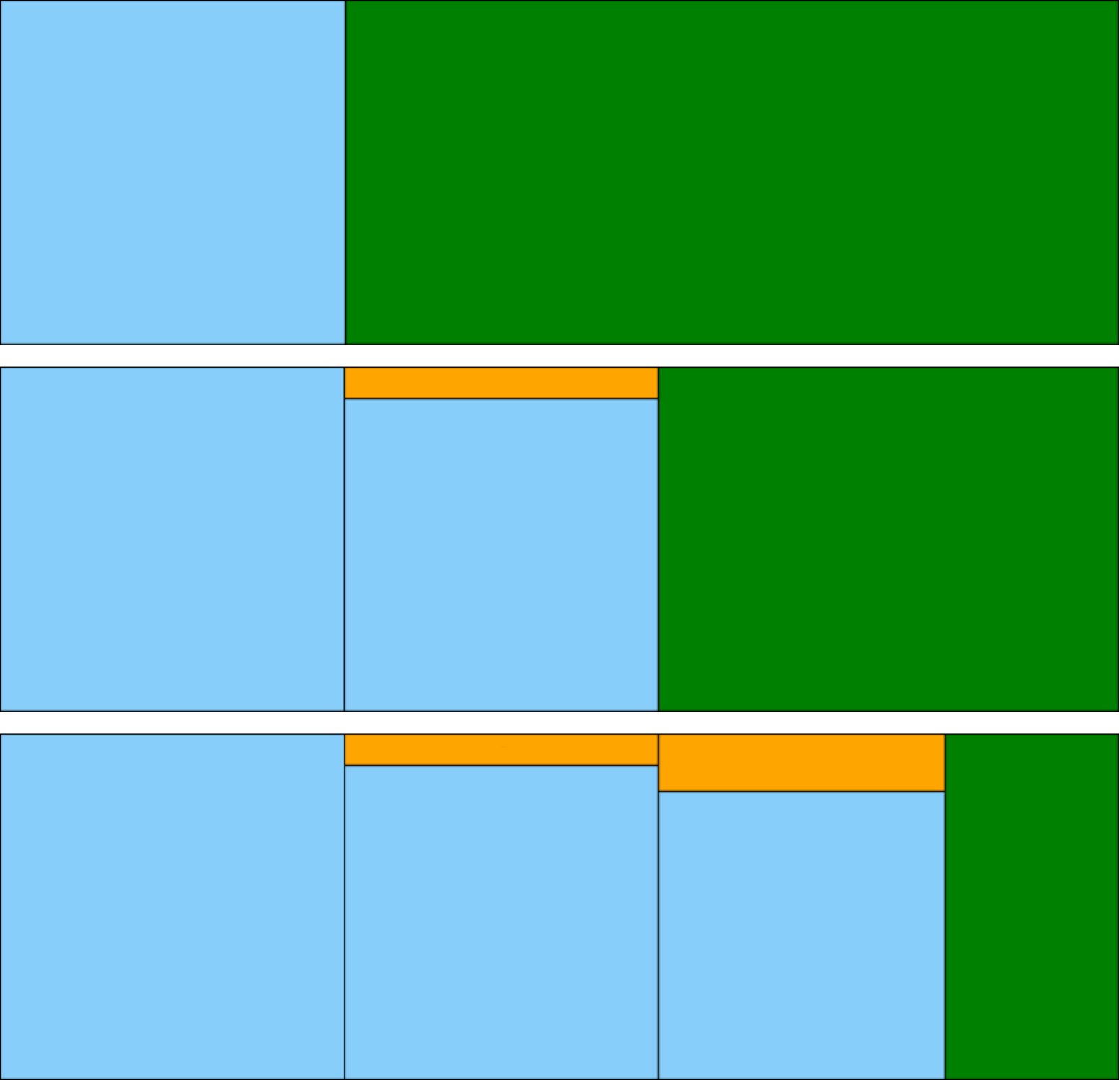}
    \caption{Illustration of the packing of a row of details in a stripe by the stacking method; obtained normal boxes are colored in yellow and endpoints are colored in green.}
    \label{fig:4}
\end{figure}

Our goal is to prove that with all sufficiently large $n_0\in\mathbb N$ we can perfectly pack details $D_n$, $n\geq n_0$, in the corresponding sheet. According to the Slack-Pack algorithm, details are being packed consecutively in ascending order of their indices, i.e., first goes the detail $D_{n_0}$, then does $D_{n_0+1}$, and so on.
If all details, whose indices are less than $t$, are packed already, and we are going to pack the next detail, whose index is $t$, then we assume that the time moment under consideration is $t$.
Denote the set of all boxes obtained by the time moment $t$ as ${\mathcal B}(t)$. It includes the set of normal boxes ${\mathcal B}_{norm}(t)$ and the set of endpoints ${\mathcal B}_{ep}(t)$. 

We treat the lesser side of any rectangular object (a detail, a sheet, a box, etc.) as its \textit{width} and the greater one as its \textit{height}. Symbols~$h$ and $w$ denote functions that define the height and width, correspondingly. For example, $h(B)$ is the height of the box~$B$, while $w(R_n)=1/(n+1)$. The symbol $S(B)$ denotes the area of the box~$B$, i.e., $S(B)=h(B)w(B)$.

The algorithm under consideration belongs to the class of guillotine cutting algorithms, i.e., each cut is being made from edge to edge. This is true both when we cut a detail from the active box (see Fig.~\ref{fig:1}) and when we do a new stripe from the LRP. 

The counterintuitive character of the Slack-Pack algorithm consists in the fact that the currently chosen active box is necessarily such that by packing a detail in it we guaranteedly get both an endpoint and a normal box. This is true, even if the first detail in a row is packed in a stripe. This distinguishes the considered approach from the classic stack packing (see~Fig.~\ref{fig:4}). We treat the empty space deliberately reserved for obtaining normal boxes and endpoints of necessary sizes as the \textit{gap}. 

\subsection{The Slack-Pack algorithm}
This algorithm allows us to perfectly pack both rectangular and square details. The input data for the algorithm is~$n_0$, i.e., the index of the first detail that has to be packed. Let the symbol $n$ denote the current time moment, and let $\gamma$ stand for the parameter used by the algorithm, $\sqrt{3/2}<\gamma<3/2$. 
We will also use the parameter $N$ to control the size of gaps between packed details.

\begin{enumerate}
\item \label{init} [Initialization]. Put $n=n_0$, $N=n_0$, and ${\mathcal B}=\emptyset$. The $LRP$ coincides with the whole sheet, while $B_{act}$ is undefined. Go to Step~\ref{choose-new-active-box}.
\item \label{check-current-active-box} [Consideration of the current active box]. If $h(B_{act})\geq h(D_n)+1/N^\gamma$, then go to Step~\ref{place-detail}. Otherwise go to Step~\ref{choose-new-active-box}.
\item \label{choose-new-active-box} [Choice of a new active box]. Put $N=n$. If ${\mathcal B}=\emptyset$, then go to Step~\ref{cut-stripe}. Otherwise calculate $B_{max} = \arg\max_{B \in {\mathcal B}} w(B)$. If $w(B_{max})\geq w(D_n)+1/n^\gamma$, then put $B_{act} = B_{max}$ and go to Step~\ref{place-detail}. Otherwise go to Step~\ref{cut-stripe}.
\item \label{cut-stripe} [Cutting off a stripe]. Fail, if $w(LRP)<h(D_n)+1/n^\gamma$. Otherwise cut from the LRP a stripe with the following sizes: $B_{str}=w(LRP) \times (w(D_n)+1/n^\gamma)$. 
The part that remains after cutting out the stripe becomes a new LRP, namely, $LRP=w(LRP) \times (h(LRP)-(w(D_n)+1/n^\gamma))$. Put $B_{act} = B_{str}$, ${\mathcal B} = {\mathcal B} \cup \{ B_{str} \}$ and go to Step~\ref{place-detail}.
\item \label{place-detail} [Packing of the detail]. Place the detail in the corner of the active box and thus obtain the endpoint and the normal box with the following sizes: $B_{ep}=w(B_{act}) \times (h(B_{act})-h(D_n))$ and $B_{norm}=h(D_n) \times (w(B_{act})-w(D_n))$. Put ${\mathcal B} = {\mathcal B} \cup \{ B_{ep}, B_{norm} \} \setminus \{ B_{act} \}$, $B_{act} = B_{ep}$, $n=n+1$ and go to Step~\ref{check-current-active-box}.
\end{enumerate}

Let us describe in detail the main stages of the algorithm operation. Initially, when the set of boxes is empty, we can get an active box only by cutting out a stripe. Having packed a row of details in it, we get first boxes. However, since due to their sizes it is still impossible to use anyone of them as an active box, further details will be packed in new stripes. 
For example, in~Fig.~\ref{fig:5}, the first 6 rows of details are packed in stripes that were cut out from the LRP. 

\begin{figure}[ht]
    \centering
    \includegraphics[width=0.5\textwidth]{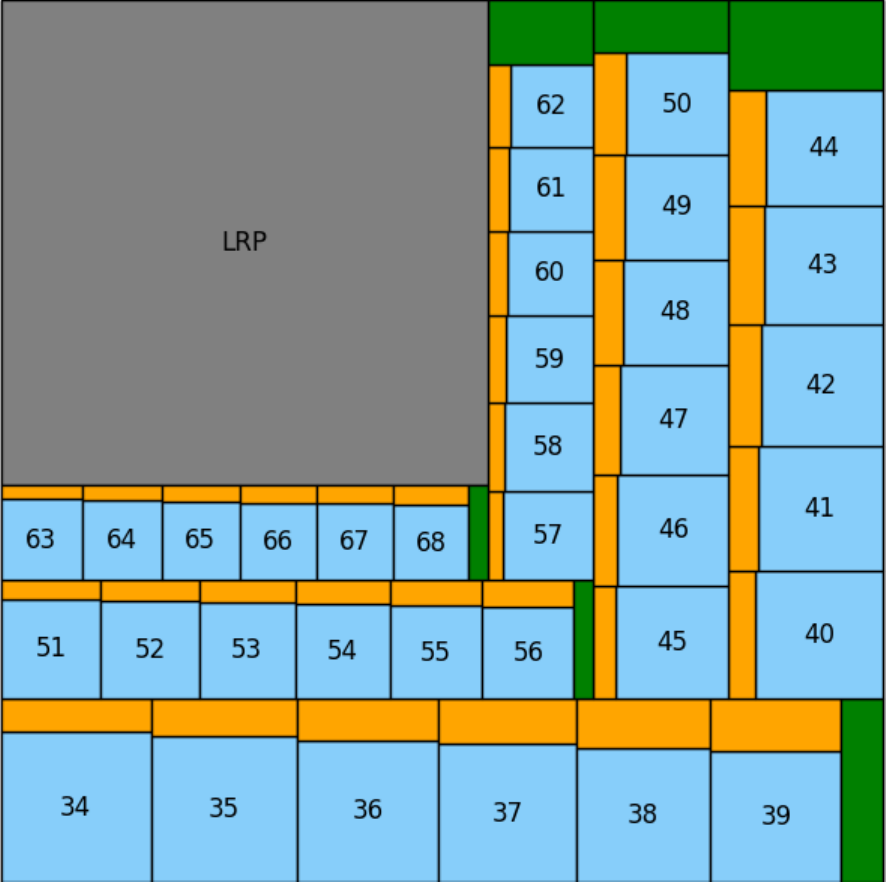}
    \caption{Illustration of the work of the Slack-Pack algorithm in packing first 35 details from the set $S_n$, $n_0=34$, $\gamma=10/7$. Endpoints and normal boxes are colored in green and yellow, correspondingly.}
    \label{fig:5}
\end{figure}

When box sizes are sufficient for packing the current detail with a necessary gap, we choose these boxes as active ones and pack the rest details in them. For example, in~Fig.~\ref{fig:6}, the first active box different from a stripe represents the endpoint obtained by packing the detail~$D_{44}$; further we pack the detail~$D_{69}$ in it.

The algorithm operation has two stages. At the first stage, we choose only stripes or endpoints as active boxes. At the second stage, we can also choose normal boxes as active ones. Fig.~\ref{fig:6} illustrates the starting moment of the second stage. The first normal box that was chosen as an active one is the box $B_{39}$, where details $D_{115}$ and $D_{116}$ were packed. Later we will prove that the second stage necessarily starts, if $n_0$ is large enough.

\begin{figure}[ht]
    \centering
    \includegraphics[width=0.5\textwidth]{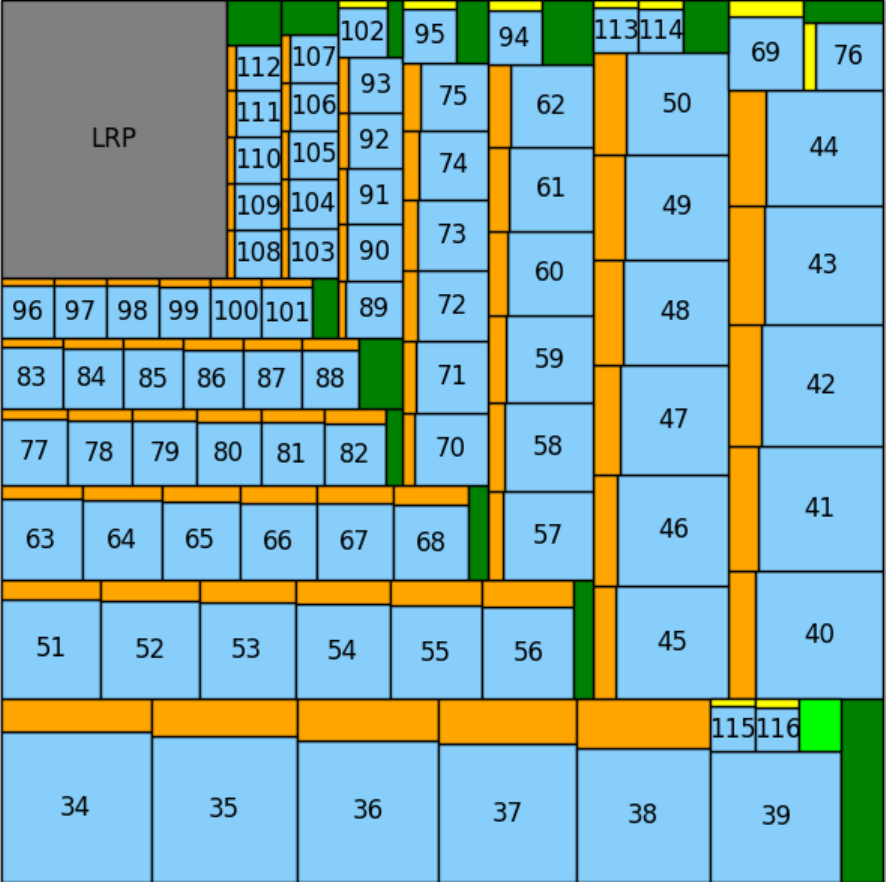}
    \caption{Illustration of the work of the Slack-Pack algorithm in packing first 83 details from the set $S_n$, $n_0=34$, $\gamma=10/7$. Endpoints and normal boxes are colored in shades of green and yellow, correspondingly.}
    \label{fig:6}
\end{figure}

\subsection{The main theorem 
}
Let the symbol~$t_0$ stand for the {\it starting moment of the second stage} (if it ever starts).
Later we will prove (see Lemma~\ref{first}) that $t_0\sim n_0^\gamma$, where the notation \textit{$\sim$ means that the limit of the ratio $t_0 / n_0^\gamma$ viewed as a function of $n_0$ tends to $1$ as $n_0 \to \infty$}.
Evidently, the onset of the time moment $t_0$ does not necessarily imply the infinity of the algorithm operation. The necessary and sufficient condition for it is the possibility of the implementation of Step~\ref{cut-stripe}, i.e., the sufficiency of sizes of the LRP for cutting the next stripe. In what follows, we also use the symbol~$\sim$ for comparing the behavior of other parameters of the algorithm with $n_0\to\infty$. 

Let us state the main results of this paper. They are valid both for rectangular and square details, therefore, in what follows, as a rule, we do not consider these two cases separately. We use the following denotations:\\
$S_{LRP}(t)$ is the area of the LRP at time moment~$t$;\\
$S_{norm}(t)$ is the area of all normal boxes at time moment~$t$; \\
$S_{ep}(t)$ is the area of all endpoints at time moment~$t$; \\
$S_{com}(t)$ is the total area of all the rest details at time moment~$t$ (in the case of details $R_n$, it equals~$1/t$).
Evidently, the following correlation is valid:
\begin{equation}
\label{eq0}
S_{LRP}(t)+S_{norm}(t)+S_{ep}(t)=S_{com}(t).
\end{equation}
The goal of this paper is to make an attempt to prove that for any arbitrarily small $\delta>0$ with sufficiently large~$n_0$ the following inequality is valid at all time moments~$t$, when Step~\ref{cut-stripe} has to be implemented:
\begin{equation}
\label{mainneq}
S_{LRP}(t)/S_{com}(t)>1-1/\gamma-\delta .
\end{equation}
We treat such time moments as {\it critical} and denote them as~$T_{crit}$. 

Inequality~\eqref{mainneq} states that $w(LRP(t)) \times h(LRP(t)) / S_{com}(t) > c > 0$ for all moments $t \in T_{crit}$. It is straightforward to show that in this case
$w(LRP(t)) = \Theta(1/\sqrt{t})$ and $h(LRP(t)) = \Theta(1/\sqrt{t})$ if $n_0$ is chosen sufficiently large. Hence, the detail~$D_t$ will successfully fit in Step~4 into the LRP since sizes of the detail are $o(1/\sqrt{t})$ and algorithm will not stop. 

This consequence of the inequality $S_{LRP}(t)/S_{com}(t) > c > 0$ is also used in other algorithms where LRP arises, therefore the study of the asymptotic behavior of the fraction $S_{LRP}(t)/S_{com}(t)$ is extremely important.

Note that at the time moments, when we have to cut a stripe from the LRP, the following inequality is fulfilled:
\begin{equation}
\label{wtcrit}
\max_{B\in{\mathcal B}(t)} w(B)<1/t+1/t^\gamma.
\end{equation}
For details $R_n$ inequality~\eqref{wtcrit} can be strengthened 
by replacing $1/t$ with $1/(t+1)$. However, for our subsequent analysis, it is more convenient to keep it in this universal form.

In what follows, we often use this inequality at time moments~$T_{crit}$. 

Let us consider two types of normal boxes. If a new active box chosen on Step~\ref{choose-new-active-box} represents a stripe, then we treat normal boxes obtained by cutting a row of details from the corresponding active box as \textit{normal boxes of the first kind} and denote them as~${\mathcal B}_{norm,1}$ (in Fig.~\ref{fig:6}, they are colored in dark yellow). But if some other box was chosen as an active one, then we treat normal boxes obtained by cutting a row of details from such an active box as \textit{normal boxes of the second kind} and denote them as~${\mathcal B}_{norm,2}$ (in Fig.~\ref{fig:6}, they are colored in light yellow). 

In the next section, we prove one important property of shapes of normal boxes, namely, 
{\it
For all $B\in{\mathcal B}_{norm,1}$, with any arbitrarily small $\sigma>0$ and sufficiently large~$n_0$ (where $n_0$ depends on $\sigma$), the following inequality is valid:
\begin{equation}
\label{neq}
h(B)^\gamma (1+\sigma) \geq w(B).
\end{equation}
}

According to experimental results, with sufficiently large~$n_0$ this correlation is valid for all boxes $B\in{\mathcal B}_{norm}$, not only for boxes of the first kind $B_{norm,1}$. Strictly speaking, we state the following assumption.
\begin{assumption}
\label{prop2}
For certain values of the algorithm parameter $\gamma\in(\sqrt{3/2},3/2)$, with any arbitrarily small $\sigma>0$ and sufficiently large~$n_0$, inequality~\eqref{neq} is fulfilled for all $B\in{\mathcal B}_{norm}$.
\end{assumption}

The following lemma is the key lemma in this paper:
\begin{lemma}
\label{keylemma}
If with certain~$\gamma$ Assumption~\ref{prop2} is fulfilled, then for any $\varepsilon>0$ with sufficiently large~$n_0$ the following inequality is valid at all time moments~$t\in T_{crit}$:
\begin{equation}
\label{SLRP} 
S_{norm}(t)/S_{com}(t)<1/\gamma+\varepsilon.
\end{equation}
 \end{lemma}
In view of this proposition, the proof of formula~\eqref{mainneq} is reduced to that of the infinite smallness of~$S_{ep}(t)/S_{com}(t)$ with sufficiently large~$n_0$.

Analogously two types of normal boxes, we consider two types of endpoints. If on Step~\ref{choose-new-active-box} a newly chosen active box represents a stripe or some other endpoint of the first kind, then we treat the endpoint obtained by cutting a row of details from the corresponding active box as an \textit{endpoint of the first kind} and denote it as~${\mathcal B}_{ep,1}$ (in Fig.~\ref{fig:6}, such endpoints are colored in dark green). But if the chosen active box represents a normal box or some other endpoint of the second kind, then we treat the endpoint obtained by cutting a row of details from this active box as an \textit{endpoint of the second kind} and denote it as~${\mathcal B}_{ep,2}$ (in Fig.~\ref{fig:6}, such endpoints are colored in light green). 

We can prove (without any assumption) that inequality~\eqref{mainneq} is fulfilled before time moment~$t_0$. Before this time moment we do not obtain boxes ${\mathcal B}_{ep,2}$ and we can get boxes ${\mathcal B}_{norm,2}$ only from endpoints. We can prove the lemma on the finiteness of the first stage of the algorithm, owing to the fact that the relative area of endpoints of the first kind $S_{ep,1}(t)/S_{com}(t)$ with sufficiently large~$n_0$ can be arbitrarily small. 

For completing the proof of the main result (under Assumption~\ref{prop2}), it remains to make sure that the value $S_{ep,2}(t)/S_{com}(t)$ is infinitesimal. However, we have succeeded in proving this property only under one additional assumption based on the hypothesis of a pseudorandom character of the behavior of endpoints.
Let us determine the mean value of the ratio of the box height to its width for boxes ${\mathcal B}_{ep,2}(t)$.

Assume that the width of the endpoint obtained at time moment~$n$ represents a random value $\xi_n$ uniformly distributed on the segment, whose left and right boundaries have, correspondingly, the order of the minimal and maximal possible width of an endpoint at time moment~$n$, while the order of the height of this endpoint is $1/n$.
Note that at time moment $t$ we have a set of endpoints obtained at time moments $n$, $n<t$. Assume that they obey a probability distribution law which agrees with all previously made assumptions.
In Appendix~\ref{secA1}, we prove that for such a mixture of uniformly distributed random values $\xi_n$, the asymptotics of the mean value of the expression $(1/n)/\xi_n$ is $(\gamma^2-1)\ln(t)$.
We also adduce results of numerical experiments which, for various values of $\gamma$, allow us to assume that the following ratio of the sample average to $\ln(t)$
$$
\frac{1}{\ln(t)\,|{\mathcal B}_{ep,2}(t)|}\sum_{B\in{\mathcal B}_{ep,2}(t)} \frac{h(B)}{w(B)}  
$$
tends to $\gamma^2-1$. 
We also demonstrate that values of the mixture parameters used by us agree with experimental data. 

Therefore, the following assumption is quite natural.
\begin{assumption}
\label{prop1}
For considered values of the parameter~$\gamma$ with sufficiently large~$n_0$ for all $t\in T_{crit}$, $t> t_0$, the following inequality is valid:
\begin{equation}
\label{eq_prop}
\frac{1}{|{\mathcal B}_{ep,2}(t)|}\sum_{B\in{\mathcal B}_{ep,2}(t)} \frac{h(B)}{w(B)} = O(\ln(t)).
\end{equation}
\end{assumption}

\begin{theorem}[The main theorem]
\label{maintheorem}
Let assumptions~\ref{prop2} and~\ref{prop1} be fulfilled for some fixed $\gamma$, $\gamma\in(\sqrt{3/2},3/2)$. Then (for this $\gamma$)
for any $\delta>0$, with sufficiently large~$n_0$ for all $t\in T_{crit}$ 
inequality~\eqref{mainneq} is fulfilled and, consequently, it is possible to obtain a perfect packing by using the Slack-Pack algorithm.
\end{theorem}

\section{Lemmas on normal boxes\label{sectiont0}}\label{sec3}
The main result of this section is the proof of the key Lemma~\ref{keylemma}. To this end, we first have to prove inequality~\eqref{neq} for normal boxes of the first kind.

\begin{lemma}
\label{norm}
For any sufficiently small $\sigma>0$ the choice of sufficiently large~$n_0$ guarantees the fulfillment of inequality~\eqref{neq} for all $B\in{\mathcal B}_{norm,1}$.
\end{lemma}

\begin{proof}
For details $R_n$, evidently, $S_{com}(n)=1/n$. Let us estimate the asymptotics of $S_{com}(n)$ for details $S_n$:
\begin{equation}
\label{scom}
S_{com}(n)=\sum_{t=n}^\infty 1/t^2\sim \int_{x=n}^\infty \frac{dx}{x^2}=1/n.
\end{equation}

Recall that $N$ is the index of the first detail in the row. Since $S_{LRP}(N)\leq S_{com}(N)$, we conclude that $w(LRP(N))=O(1/\sqrt{N})$. Denote the number of details in this row by $k$, $k=o(N)$. Let us estimate the sum of heights of these details:
\begin{equation}
\label{kn}
\frac{1}{N}+\frac{1}{N+1}+\ldots+\frac{1}{N+k-1}\sim \ln\left( \frac{N+k}{N} \right) \sim \frac{k}{N}.
\end{equation}
Hence we conclude that $w(LRP(N))=O(1/\sqrt{N})\sim \frac{k}{N}$, i.e., $k=O(\sqrt{N})$.

Note that having packed a row of~$k$ details, we get~$k$ new normal boxes, whose heights equal $1/N,\ldots 1/(N+k-1)$ and widths are not less than $N^{-\gamma}$. Therefore, for all $B\in{\mathcal B}_{norm}$,
\begin{equation}
\label{neqInv}
h(B)^\gamma \leq w(B).
\end{equation}

Consider the normal box $B_n$: $w(B_n)\sim\frac{1}{N^{\gamma}} + \frac{1}{N} - \frac{1}{n}$, $h(B_n)=\frac{1}{n}$; here
\begin{equation}
\label{sqrtm}
n=N+O(\sqrt{N}).
\end{equation}

In view of the constraint $\gamma<3/2$ we get the equality
$$
\lim\limits_{n\to\infty} \frac{w(B_n)}{h(B_n)^{\gamma}}=
\lim\limits_{N\to\infty}  \Bigl( \frac{1}{N}+\frac{1}{N^{\gamma}} - \frac{1}{N + O(\sqrt{N})} \Bigr)\times
\Bigl( N + O(\sqrt{N})\Bigr)^\gamma  = 1.
$$
\end{proof}

\begin{remark}
According to results of a more accurate analysis, the maximal value of $\frac{h(B_n)^\gamma}{w(B_n)}$ for $B_n\in{\mathcal B}_{norm,1}$ is attained at the end of the first row in the normal box. 
\end{remark}

We do not prove this remark, because we do not use it when proving the main propositions. Let us now prove the key Lemma~\ref{keylemma}.

\begin{proof}[Proof of Lemma~\ref{keylemma}]
Formula~\eqref{wtcrit} implies that for any $\delta>0$ with $n_0>({1}/{\delta})^{\gamma-1}$ the following inequality is valid for all $B\in{\mathcal B}(t)$, $t\in T_{crit}$: 
\begin{equation}
\label{condw}
w(B)<(1+\delta)/t.
\end{equation}
In particular, this correlation is fulfilled for $B_n\in{\mathcal B}_{norm}(t)$. Formulas~\eqref{neqInv} and~\eqref{condw} imply that for all $B_n\in{\mathcal B}_{norm}(t)$ with $t\in T_{crit}$ the index~$n$ satisfies the inequality
\begin{equation}
\label{ngamma}
n^\gamma (1+\delta) > t.
\end{equation}
Let us make use of Assumption~\ref{prop2}. In accordance with inequality~\eqref{neq} we get the correlation
\begin{eqnarray*}
S_{norm}(t)&=&\sum_{B\in{\mathcal B}_{norm}(t)} w(B)h(B) \leq (1+\sigma) \sum_{B\in{\mathcal B}_{norm}(t)} h(B)^{1+\gamma}=\\
&=& (1+\sigma)\sum_{n: B_n\in{\mathcal B}_{norm}(t)} n^{-1-\gamma},
\end{eqnarray*}
where $\sigma$ is an arbitrarily small positive value, while $n_0$ is sufficiently large.
By applying formula~\eqref{ngamma} we get the correlation
$$
\frac{S_{norm}(t)}{1+\sigma}\leq \sum_{(t/(1+\delta))^{1/\gamma}<n<t} n^{-1-\gamma}\sim\int_{(t/(1+\delta))^{1/\gamma}}^t x^{-1-\gamma}\,dx\sim \frac{1+\delta}{\gamma t},
$$
which implies the validity of inequality~\eqref{SLRP} for arbitrarily small~$\varepsilon>0$. 
\end{proof}

Note that if instead of boxes ${\mathcal B}_{norm}$ we consider those ${\mathcal B}_{norm,1}$, then we can refer to Lemma~\ref{norm} rather than to Assumption~\ref{prop2}, and then the key lemma takes the form of the following unconditional assertion.
\begin{corollary}
\label{corol1}
For any $\varepsilon>0$, with sufficiently large~$n_0$ the following inequality is valid at all time moments~$t\in T_{crit}$: 
$$
S_{norm,1}(t)/S_{com}(t)<1/\gamma+\varepsilon.
$$
\end{corollary}

\section{The rest auxiliary lemmas}\label{sec4}
The main result of this section is the proof of two important lemmas. In the first part, we estimate the percentage of endpoints of the first kind.
\begin{lemma}
\label{lemma_Sep1}
For any $\varepsilon>0$ with sufficiently large~$n_0$ the following inequality is valid at all time moments~$t\in T_{crit}$: 
$${S_{ep,1}(t)/S_{com}(t)<\varepsilon.}$$
\end{lemma}
In the second part of this section, we consider the time moment, when the first stage of the algorithm passes into the second one. We prove that the second stage necessarily starts and estimate the corresponding time moment. To this end, we prove that inequality~\eqref{mainneq} is fulfilled at the first stage of the algorithm.
\begin{lemma}
\label{t0exist}
For any $\delta>0$ with sufficiently large~$n_0$ inequality~\eqref{mainneq} is fulfilled for all $t\leq t_0<\infty$, $t\in T_{crit}$. This means that the algorithm, at least, will not stop at the first stage.
\end{lemma}

\subsection{Lemmas on the ratio of endpoints of the first kind}
Let the symbol ${\mathcal B}'(t)$ stand for the set of all endpoints obtained before time moment~$t$, i.e., ${\mathcal B}'(t)=\bigcup\limits_{n=n_0}^t{\mathcal B}(n)$.

Recall that endpoints of the first kind can be obtained by packing a row of details in the active box, which represents a stripe or some other endpoint of the first kind. Denote by ${\mathcal B}_{s.ep,1}(t)$ the set of all endpoints of the first kind obtained in a stripe up to time moment~$t$. Then ${\mathcal B}'_{s.ep,1}(t)=\bigcup\limits_{n=n_0}^t{\mathcal B}_{s.ep,1}(n)$.

Before proving Lemma~\ref{lemma_Sep1} let us consider the following geometric lemma, which is important for further reasoning.
\begin{lemma}
The following inequality is valid:
\label{lemma_Sep1init}
$$\sum_{B\in {\mathcal B}'_{s.ep,1}(t)} h(B)<p,\quad \mbox{where $p$ is the half-perimeter of the initial sheet.}$$
\end{lemma}
\begin{proof}
One may cut out stripes from the LRP either ``along the width'' or the height of the initial sheet. Having packed a row of details, we get an endpoint in each stripe. We can easily make sure that the sum of heights of endpoints under consideration equals the sum of widths of all stripes cut off by time moment~$t$. Let the symbol $h_1$ stand for the sum of heights of endpoints obtained from stripes that were cut out ``along the width'' of the sheet, and let $h_2$ do for those that were cut out ``along its height'' (see~ Fig.~\ref{fig:7}). Hence, 
$$
\sum_{B\in {\mathcal B}'_{s.ep,1}(t)} h(B)=h_1+h_2=p- w(LRP(t))-h(LRP(t)),
$$
here $w(LRP(t))+h(LRP(t))$ is the half-perimeter of the LRP at time moment $t$. This, evidently, implies the desired lemma.
\end{proof}

\begin{figure}[h]
    \centering
    \includegraphics[width=0.5\textwidth]{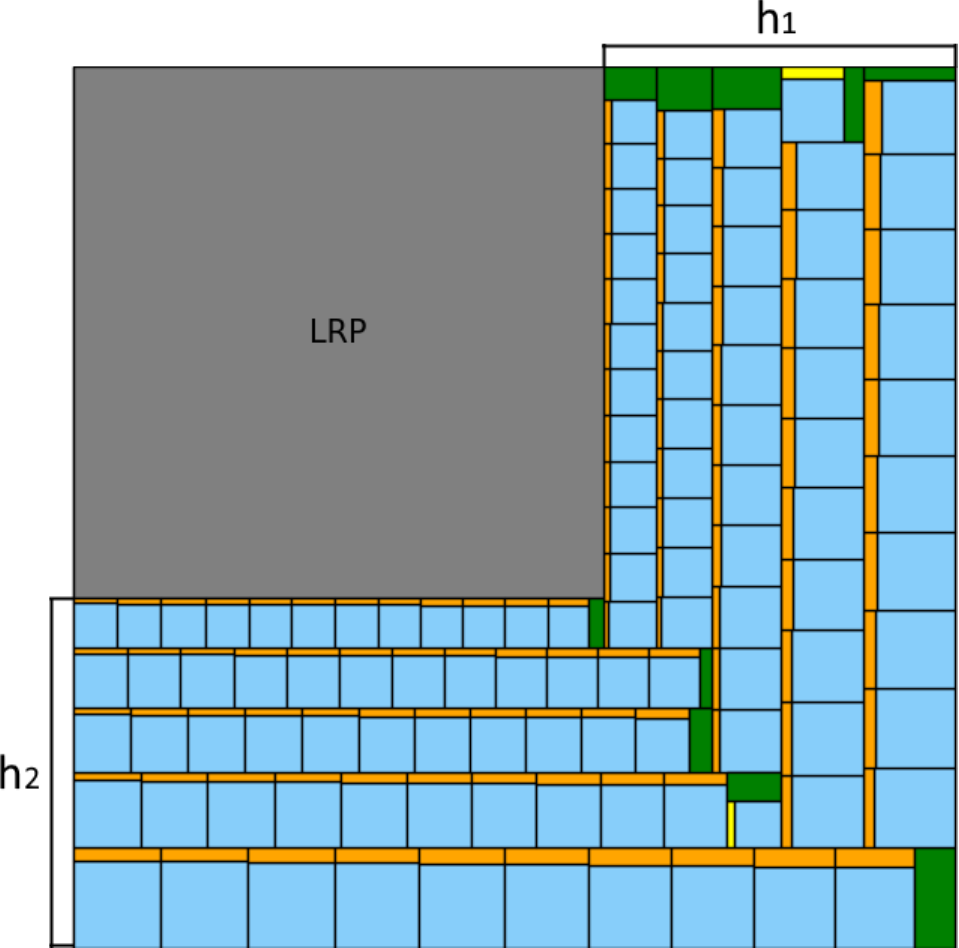}
    \caption{Illustration of Lemma~\ref{lemma_Sep1init} and of Lemma~\ref{lemma_Sep1}.}
    \label{fig:7}
\end{figure}

\begin{proof}[Proof of Lemma~\ref{lemma_Sep1}]
If the active box represents an endpoint of the first kind, then by packing a row of details in it we get a new endpoint of the first kind. Evidently, its height does not exceed the height of the initial endpoint (see~ Fig.~\ref{fig:7}). By Lemma~\ref{lemma_Sep1init} we conclude that
$$\sum_{B\in {\mathcal B}_{ep,1}(t)} h(B)\leq \sum_{B\in {\mathcal B}'_{s.ep,1}(t)} h(B)<p,$$
whence
$$
S_{ep,1}(t)=\sum_{B\in {\mathcal B}_{ep,1}(t)} w(B)h(B)< p \max_{B\in{\mathcal B}_{ep,1}(t)} w(B).
$$
According to formula~\eqref{condw}, $\max_{B\in{\mathcal B}_{ep,1}(t)}w(B)<(1+\delta)/t$ and according to~\eqref{scom}, $S_{com}(t)\sim 1/t$. With sufficiently large~$n_0$ the half-perimeter $p$ of the sheet can be arbitrarily small, which means that the lemma is valid.
\end{proof}

\subsection{Lemmas on the starting time moment of the second stage of the algorithm}
For proving Lemma~\ref{t0exist} we need one more corollary of Lemma~\ref{lemma_Sep1init}.
\begin{lemma}
\label{32}
With $n_0\to\infty$ at all time moments~$t\leq t_0$, 
$S_{norm,2}(t)+S_{ep,1}(t)=O(n_0^{-3/2})$. 
\end{lemma}
\begin{proof}
Choosing an endpoint of the first kind as the active box and then packing a row of details in it, we get a new endpoint of the first kind and normal boxes of the second kind.\footnote{For example, in Fig.~\ref{fig:6}, by packing the row of details $D_{113}-D_{114}$ in the endpoint above detail $D_{50}$ we get a new endpoint of the first kind and 2 normal boxes of the second kind.} Evidently, the total area of new boxes is less than the area of the initial endpoint. Since with~$t\leq t_0$ normal boxes of the second kind can be obtained only by packing details in endpoints of the first kind, we conclude that
$$S_{norm,2}(t)+S_{ep,1}(t)\leq \sum_{B\in {\mathcal B}'_{s.ep,1}(t)} S(B).$$
By Lemma~\ref{lemma_Sep1init},
$$
\sum_{B\in {\mathcal B}'_{s.ep,1}(t)} h(B)<p=2 \sqrt{S_{com}(n_0)}=O(n_0^{-1/2}).
$$
For any~$B$, evidently, $w(B)\leq n_0^{-1}$. Consequently,
$$
S_{norm,2}(t)+S_{ep,1}(t)=\sum_{B\in {\mathcal B}'_{s.ep,1}(t)} h(B) w(B)=O(n_0^{-3/2}).
$$ 
\end{proof}

\begin{proof}[Proof of Lemma~\ref{t0exist}]
Consider the first stripe cut out from the LRP.\footnote{In Fig.~\ref{fig:6}, this is the stripe, where the row of details $D_{34}-D_{39}$ is packed.} Let the symbol~$k_0$ stand for the \textit{quantity of details in this stripe} and let~${\widetilde B}$ do for the normal box obtained by packing the last detail in this row.\footnote{In Fig.~\ref{fig:6}, this is the normal box located above detail $D_{39}$.} Then the index of this normal box is $n_0+k_0-1$. Evidently, the second stage will start not later the time moment, when the row of details is packed in~${\widetilde B}$.

Let us prove that $k_0\geq \lfloor{\sqrt{n_0}}\rfloor-1$. Let the symbol~$a$ stand for the sum of heights of details from $D_{n_0}$ to $D_{n_0+\lfloor{\sqrt{n_0}}\rfloor-2}$. Since
$$
a=\frac{1}{n_0}+\ldots+\frac{1}{n_0+\lfloor{\sqrt{n_0}}\rfloor-2}<\frac{\lfloor{\sqrt{n_0}}\rfloor-1}{n_0},
$$
we conclude that $a+n_0^{-\gamma}<1/\sqrt{n_0}\leq \sqrt{S_{com}(n_0)}$, i.e., $\lfloor{\sqrt{n_0}}\rfloor-1$ first details are packed in the first stripe with the necessary gap for the endpoint.

Let the symbol $b$ stand for the sum of heights of details in the first stripe, i.e., those from $D_{n_0}$ to $D_{n_0+k_0-1}$. Formula~\eqref{kn} implies the correlation
$$
b=\frac{1}{n_0}+\ldots+\frac{1}{n_0+k_0-1}\sim \frac{k_0}{n_0},\quad \mbox{while $b<\sqrt{S_{com}(n_0)}\sim 1/\sqrt{n_0}$.}
$$

Therefore, $k_0\sim \sqrt{n_0}$ and the difference of widths~$\Delta$ of the first detail in the stripe and the last one is $\Delta\sim 1/n_0-1/(n_0+k_0)\sim n_0^{-3/2}$.
Let the symbol $t'$ stand for the first time moment, when the next detail can be packed in the box ${\widetilde B}$ with the necessary gap, whose order of magnitude is~$ (t')^{-\gamma}$. Then $1/t'+(t')^{-\gamma}\geq (1/n_0)^\gamma+\Delta$, and since $\gamma^2>3/2$, we conclude that $t'\sim n_0^{\gamma}$.

Let us prove that the algorithm does not stop earlier the starting moment of the second stage (it begins at the time moment, when ${\widetilde B}$ or some other normal box is chosen as the active box). Evidently, for any $t\in T_{crit}$, at the first stage, $t\leq t_0$, the next detail cannot be packed in ${\widetilde B}$, i.e., $t\leq t'$.

Since ${\mathcal B}_{ep,2}(t)=\emptyset$ with $t\leq t_0$, equation~\eqref{eq0} takes the form
$$
S_{LRP}(t)+S_{norm,1}(t)+S_{norm,2}(t)+S_{ep,1}(t)=S_{com}(t).
$$
According to formula~\eqref{scom}, $S_{com}(t)\sim 1/t$ and by condition, $\gamma<3/2$. Therefore, with $t=O(n_0^\gamma)$ in view of Lemma~\ref{32} we conclude that
$$(S_{norm,2}(t)+S_{ep,1}(t))/S_{com}(t)=o(1),\quad  \mbox{as $n_0\to\infty$}.
$$
Taking into account Corollary~\ref{corol1}, we conclude that with sufficiently large~$n_0$ inequality~\eqref{mainneq} is valid for all $t\in T_{crit}$, $t\leq t_0$.
\end{proof}

Let us now estimate the time moment~$t_0$, when the second stage starts.
\begin{lemma}
\label{first}
With any sufficiently large~$n_0$, $t_0< (1+\varepsilon) n_0^\gamma$, for any $\varepsilon>0$.
\end{lemma}
\begin{proof}
The starting time moment of the second stage $t_0$ can be somewhat larger than~$t'$ for two reasons. First, there might be a delay of order $O(\sqrt{t'})$ caused by the cut-off stripe still being filled with details. The second possible reason is the existence of a large number of endpoints, whose width exceeds~$w({\widetilde B})$.
Since the order of magnitude of the value added due to the first reason is only $O(\sqrt{t'})$, which is asymptotically neglectable in comparison with $t'$, it remains to consider only the second reason.

Assume that we can pack in endpoints only details with indices ranging from $t'$ to $C t'$, where $C>1$ is some fixed number independent of $t'$. 
But then the total area of these boxes is asymptotically not less than
$$
\int_{t'}^{Ct'} dx/x^2=(1-1/C)/t'.
$$
Since $t'\sim n_0^\gamma$, this inequality contradicts Lemma~\ref{32}.
\end{proof}

\begin{remark}
Let Assumption~\ref{prop2} be fulfilled. Then we can reformulate the proposition about the value~$t_0$ in Lemma~\ref{first} as $t_0\sim n_0^\gamma$.
\end{remark}
Really, inequality~\eqref{neq} is equivalent to the fact that the box~$B_n$ satisfies the inequality $(1+\sigma)/n^\gamma > w(B_n)$, therefore the width of any normal box is less than $(1+\sigma)/n_0^\gamma$, whence we conclude that $t_0>n_0^\gamma/(1+\sigma)$.
Two-sided constraints imply the correlation $t_0\sim n_0^\gamma$.

\section{Proof of the main theorem}\label{sec5}
\begin{lemma}
\label{Bntgamma}
Assume that details $D_t,D_{t+1},\ldots$ are packed in the normal box $B_n$ and Assumption~\ref{prop2} is fulfilled.
Then $t\sim n^\gamma$.
\end{lemma}
\begin{proof} 
Note that Assumption~\ref{prop2} is equivalent to the correlation $h(B)^\gamma \sim w(B)$, for all $B\in B_{norm}$. 
Really, according to the algorithm description, the inequality $w(B) \geq h(B)^\gamma$ is valid for all $B\in B_{norm}$. The inverse asymptotic nonstrict inequality is defined by Assumption~\ref{prop2}.
Hence we get the correlation 
$$w(B_t)\sim h(B_t)^\gamma\sim 1/t^\gamma=o(1/t).$$
In addition, $w(D_t)\sim 1/t$, while $w(D_t)+w(B_t)=w(B_n)$. Consequently,
$$1/t\sim w(D_t)\sim w(B_n) \sim h(B_n)^\gamma\sim 1/n^\gamma,$$
i.e., $t\sim n^\gamma$.
\end{proof}
\begin{lemma}
\label{finallemma}
Let assumptions~1 and~2 be fulfilled. Then for all $t\in T_{crit}$, $t>t_0$, $S_{ep,2}(t)=O(t^{-2+1/\gamma}\ln(t))$ as $n_0\to\infty$.
\end{lemma}
\begin{proof} 
Using inequality~\eqref{condw}, we conclude that with considered values of~$t$,
$$
S_{ep,2}(t)= \sum_{B\in{\mathcal B}_{ep,2}(t)} w^2(B) \frac{h(B)}{w(B)} = O\bigl(\frac{1}{t^2}\bigr)
 |{\mathcal B}_{ep,2}(t)|\times \frac{1}{|{\mathcal B}_{ep,2}(t)|}\sum_{B\in{\mathcal B}_{ep,2}(t)} \frac{h(B)}{w(B)}.
$$
In accordance with Assumption~\ref{prop1}, the latter cofactor is $O(\ln (t))$. 
We can estimate the number of endpoints of the second kind $|{\mathcal B}_{ep,2}|$ 
with the help of Lemma~\ref{Bntgamma}. We get any endpoint of the second kind when we initially pack a row of details in a normal box and, possibly, when we further pack details in obtained endpoints. Then any box $B\in{\mathcal B}_{ep,2}(t)$ at a fixed time moment~$t$ corresponds to exactly one normal box $B_n$ used at a certain time moment $\tau\leq t$ as the active box. By Lemma~\ref{Bntgamma}, $\tau\sim n^\gamma$. Therefore, the quantity of normal boxes, whose indices~$n$ satisfy the given asymptotic equality, is $O(t^{1/\gamma})$, whence we get the correlation
$$
S_{ep,2}(t)=O(t^{-2}) O(t^{1/\gamma}) O(\ln (t)).
$$
\end{proof}
\begin{proof}[Proof of Theorem~\ref{maintheorem}]
The following correlation is valid:
\begin{equation}
\label{finaleq}
S_{LRP}(t)+S_{norm}(t)+S_{ep,1}(t)+S_{ep,2}(t)=S_{com}(t).
\end{equation}
In addition, according to formula~\eqref{scom}, $S_{com}(t)\sim 1/t$. 
For $t\leq t_0$ the assertion of the theorem is valid by Lemma~\ref{first}, and there is no need in using assumptions~\ref{prop2} and~\ref{prop1}. 
Let us now consider the general case for all $t\in T_{crit}$.
Let us divide the left- and right-hand sides of equality~\eqref{finaleq} by $S_{com}(t)$.
The second addend satisfies inequality~\eqref{SLRP}, the third one does Lemma~\ref{lemma_Sep1}, the fourth one, according to Lemma~\ref{finallemma}, satisfies the bound $S_{ep,2}(t)/S_{com}(t)=o(1)$. 
\end{proof}   

\section{Conclusion~\label{concl}}
The graphs given below (see~Fig.~\ref{pic:1} and~Fig.~\ref{pic:2}) illustrate the dependence of the ratio of LRP on the value of the parameter $n_0$ with $\gamma=4/3$, when the Slack-Pack algorithm is being used for packing up to $10^{10}$ details.
The horizontal dashed line corresponds to the value of $1-1/\gamma$, while the vertical dash-dotted line does to the argument $n_0^\gamma\sim t_0$.
According to experimental results, after attaining the value $t = t_0$ the ratio of LRP stabilizes at the level slightly exceeding $1 - 1/\gamma - \delta$, where $\delta$ is some small value.

\begin{figure}[ht]
\begin{center}
\begin{tabular}{p{0.45\textwidth}p{0.45\textwidth}}
\includegraphics[width=0.45\textwidth]{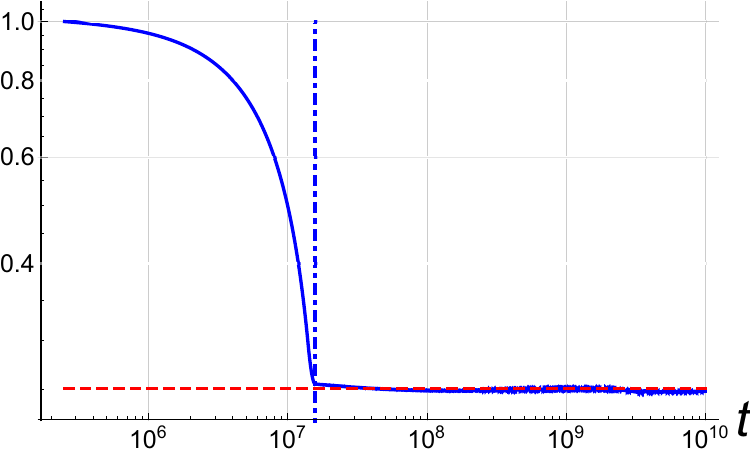} & \includegraphics[width=0.45\textwidth]{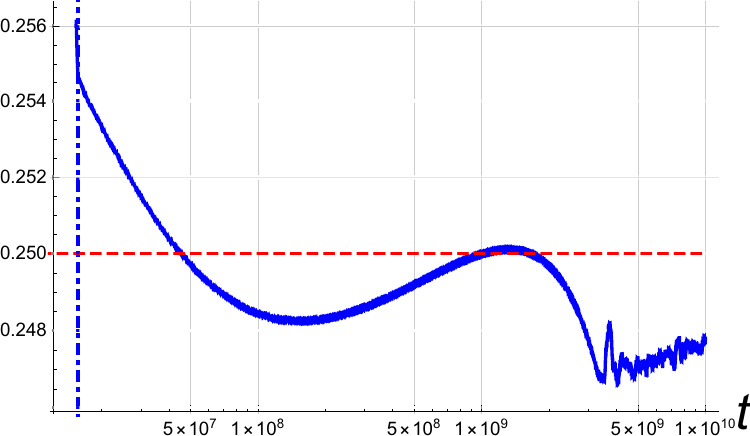} \\
\end{tabular}
\caption{\label{pic:1} The behavior of $S_{LRP}(t)/S_{com}(t)$ for sets of details $R_n$ with $\gamma=4/3$, $n_0=500^2$. Here $1-1/\gamma=0.25$, $n_0^\gamma\approx 1.5749\times 10^7$.}
\end{center}
\end{figure}
\begin{figure}[ht]
\begin{center}
\begin{tabular}{p{0.45\textwidth}p{0.45\textwidth}}
\includegraphics[width=0.45\textwidth]{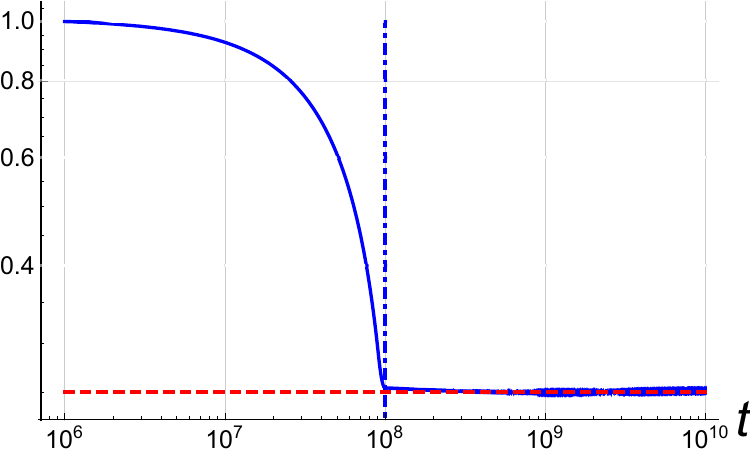} & \includegraphics[width=0.45\textwidth]{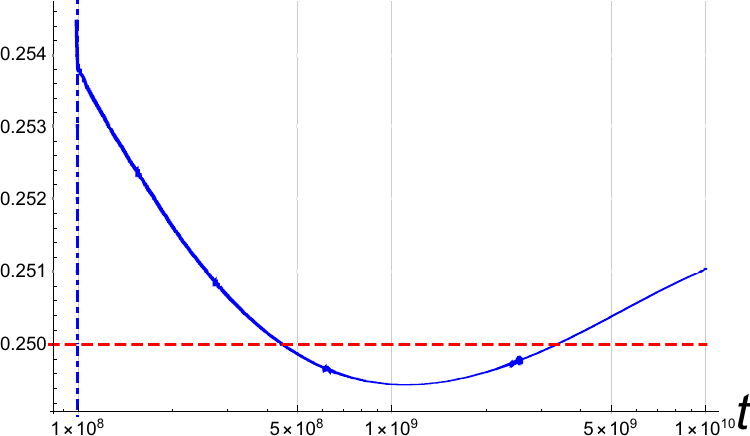} \\
\end{tabular}
\caption{\label{pic:2} The behavior of $S_{LRP}(t)/S_{com}(t)$ for sets of details $R_n$ with $\gamma=4/3$, $n_0=1000^2$. Here $1-1/\gamma=0.25$, $n_0^\gamma=10^8$.}
\end{center}
\end{figure}

We have also performed certain experiments for verifying inequality~\eqref{neq} in Assumption~\ref{prop2}. For the value $1+\sigma$ we used the expression $h({\widetilde B})^\gamma / w({\widetilde B})$, where ${\widetilde B}$ denotes the last normal box in the first stripe. 
This choice is motivated by the fact that among all normal boxes of the first kind this particular box yields the smallest value of the ratio $h(B)^\gamma / w(B)$.
Experiments were performed for sets of details $R_n$ with the same values of $n_0$ and $\gamma$ as those used in previous experiments. According to obtained results, inequality~\eqref{neq} is fulfilled for all considered normal boxes.
See also the Appendix~\ref{secA1} for theoretical and experimental arguments in favor of Assumption~\ref{prop1}. 

We hope for succeeding in complete theoretical justification of our assumptions or their modified versions. The performed experiments (results of some of them are adduced above) confirm the main proposition of this paper, namely, the fact that the ratio of LRP possesses a predictable asymptotic behavior. Just this property essentially distinguishes our algorithm from those considered previously.

\vspace{4pt}
\noindent{\bf Acknowledgements.}
The authors is grateful to Olga Kashina for her help in the translation of this paper.  

\vspace{4pt}
\noindent{\bf Data Availability.} 
The data is publicly available in public repositories. See \href{https://github.com/AlexeyKislovskiy/Slack-Pack-algorithm}{https://github.com/AlexeyKislovskiy/Slack-Pack-algorithm} for the code of the program that visualizes the work of various packing algorithms. See also the code of the program in \href{https://github.com/sensen65/Slack-Pack-algorithm-Meir-Moser}{https://github.com/sensen65/Slack-Pack-algorithm-Meir-Moser} that implements the Slack-Pack  algorithm for $10^{10}$ details.

\subsection*{Declarations}

\noindent{\bf Conflict of interest.}
The authors have no conflicts of interest relevant to the content of this article.

\begin{appendices}

\section{About Assumption~\ref{prop1}}\label{secA1}

In this appendix, we theoretically study the asymptotic behavior of the mean value to of the ratio between the height and the width of endpoints of the second kind at time moment $t$, provided that the width of boxes obtained earlier at time moment~$n$ represents a random value uniformly distributed on a segment, which is defined by the algorithm (from $1/n^\gamma$ up to the value that equals the height of boxes, whose asymptotics is $1/n$).

Note that the quantity of endpoints of the second kind is being accumulated, because each cut of an endpoint of the second kind gives a new endpoint of the second kind. Therefore, their quantity coincides with the number of normal boxes that were cut. Assume that this number at time moment~$n$ asymptotically equals the value indicated in Lemma~\ref{Bntgamma} in the main text, namely, $n^{1/\gamma}$.

Completing our research, we compare theoretical qualitative results with experimental data that illustrate the behavior of the mean value of the ratio between the height and width of endpoints of the second kind with various values of the parameter~$\gamma$. In addition, we compare graphs that demonstrate the quantity of endpoints of the second kind with our theoretical estimates.

In accordance with propositions in the main part of the paper, we assume that all endpoints that exist at time moment~$n$ and are such that their width exceeds $1/t$ are already used by time moment~$t$. Therefore, widths of boxes obtained at time moment~$n$, which remain non-cut by time moment~$t$, are assumed to be random values uniformly distributed on the segment $[1/n^\gamma, 1/t]$.
In this research, we replace positive integer values $n$ with a continuous parameter, which takes on values in the corresponding range. Evidently, the inequality $1/n^\gamma < 1/t < 1/n$ has to be fulfilled. It is equivalent to the following constraint imposed on $n$: 
\begin{equation}
\label{eqn_on_n}
t^{1/\gamma}<n<t.
\end{equation}

\begin{lemma}
\label{appen1}
Assume that $\gamma>1$, $n^\gamma>t>1$, and $\xi_n$ is a random value uniformly distributed on the segment $[1/n^\gamma,1/t]$. 
Then the mean value of the random value $\eta_n=1/(n\xi_n)$ equals
$$
\mathbb{E} \eta_n =\frac{t\, n^{\gamma -1} \left(\ln (n^{\gamma})-\ln(t)\right)}{n^{\gamma }-t}.
$$
\end{lemma}

\begin{proof}
The cummulative density function $F(x)$ of the random value $\xi_n$, which is uniformly distributed on the segment $[1/n^\gamma,1/t]$, takes (on this segment) the following form: 
$$
F(x)=\frac{t \left(x\, n^\gamma-1\right)}{n^\gamma-t}. 
$$
Therefore, the probability density function $f(y)$ of the random value $1/\xi$ is nontrivial on the segment $[t,n^\gamma]$ and takes on it the following form:
$$
f(y)=\frac{d}{dy}(1-F(1/y))=\frac{t\,n^{\gamma }}{y^2 \left(n^{\gamma}-t\right)}.
$$
By integrating the function $f(y)y/n$ over the segment $[1/n^\gamma,1/t]$, we get the assertion of the lemma.
\end{proof}

We have estimated the mean value of the ratio between the height and the width of endpoints obtained at time moment~$n$, which were non-cut by time moment~$t$.
All these values of $n$ satisfy condition~\eqref{eqn_on_n}. However, at time moment~$t$, we have endpoints which were obtained at various time moments~$n$. Therefore, we are interested in the mean value of a mixture of random values. For calculating it, we need to estimate the probability density function for endpoints which were obtained at time moment~$n$ and remained non-cut by time moment~$t$. 
 
Assume that the total number of endpoints of the second kind at time moment~$n$, in accordance with Lemma~\ref{Bntgamma}, equals $n^{1/\gamma}$. Moreover, widths of endpoints obtained at time moment~$n$ are uniformly distributed on the segment $[1/n^\gamma, 1/n]$. The percentage of these endpoints, whose widths are less than $1/t$, is defined by the function $G(n,t)$, which takes the form
$$
G(n,t)=\frac{1/t-1/n^\gamma}{1/n-1/n^\gamma}.
$$
Only these endpoints remain non-cut by time moment~$t$.
Therefore, in accordance with our assumptions, the probability distribution function of the parameter $n$ for random values $\xi_n$ obeys the formula
$$
G(n,t) \frac{n^{1/\gamma}}{t^{1/\gamma}},
$$
where $n$ satisfies condition~\eqref{eqn_on_n}, while $t^{1/\gamma}$ is a normalizing constant.
Correspondingly, the probability density function for the mixture of random values $g(n)$ takes the form
$$
g(n)=\frac{d}{t^{1/\gamma} dn} \left(  G(n,t) n^{1/\gamma} \right). 
$$

\begin{lemma}
\label{appen2}
Let assumptions of Lemma~\ref{appen1} be fulfilled. Denote by the symbol $e(t)$ the mean value of the mixture of random values $\eta_n$ with the probability density function $g(n)$, $t^{1/\gamma}<n<t$.
Then 
$$
\lim_{t\to\infty} \frac{e(t)}{\ln(t)}=\gamma^2-1.
$$
\end{lemma}
\begin{proof}
We need to calculate the asymptotics of the integral
$$
\int_{t^{1/\gamma}}^t g(n)\, \mathbb{E} \eta_n \,dn
$$
with $t\to\infty$. For convenience, we change variables in the integrand:
$$
z = \frac{n^\gamma - t}{t}, \quad \text{i.e.,} \quad n = (t(1+z))^{1/\gamma}, \quad dn = \frac{t (t(1+z))^{1/\gamma - 1}}{\gamma} \, dz.
$$ 
Then we need to integrate the integrand $r(z,t)$ with respect to $z$ varying from 0 to nearly $t^{\gamma - 1}$ (more precisely, to $t^{\gamma - 1} - 1$).
One can easily make sure that the function $r(z,t)$ is continuous in $z$ on the integration interval and
$$\lim_{z\to 0} r(z,t)=\frac{t^{\frac{1}{\gamma ^2}-\frac{1}{\gamma}+1}}{t-t^{\frac{1}{\gamma }}}.
$$
Therefore, when calculating the limit value of $e(t)$, we can neglect the value of the integral of $r(z,t)$ with respect to $z\in [0,o(t^{\frac{1}{\gamma}-\frac{1}{\gamma^2}})]$ . 
Performing the integration over the remaining part of the integration domain, we can replace $z$ in the integrand with $z + 1$; this does not affect the asymptotics.
After this transform we conclude that the asymptotics of $e(t)$ coincides with that of the integral
$$
I=\int_0^{t^{\gamma-1}} \frac{(\gamma +1) t^{-1/\gamma } (t
   (z+1))^{\frac{1}{\gamma ^2}}\ln (z+1)}{\gamma ^2 (z+1)}\, dz,
$$
whose value can be calculated explicitly. It equals
$$
I=(\gamma +1)\, t^{-1/\gamma } \left(t^{\gamma
   }+t\right)^{1/\gamma ^2} \left(\gamma ^2
   \left(\left(\frac{t}{t^{\gamma
   }+t}\right)^{1/\gamma ^2}-1\right)+\ln
   (t^{\gamma -1}+1)\right).
$$
The first term in the latter braces is $O(1)$, and we can neglect it. Then 
$$
I\sim (\gamma +1)\,t^{-1/\gamma } \left(t^{\gamma
   }+t\right)^{1/\gamma ^2}\ln
   (t^{\gamma -1})\sim (\gamma^2-1)\ln (t).
$$
\end{proof}

\begin{remark}
\label{appnote1}
By our assumption, the quantity of endpoints of the second kind, which exist at time moment~$n$ equals $n^{1/\gamma}$. 
At the same time, at time moment~$t$, the quantity of those endpoints among them that remain non-cut is only 
$$
G(n,t) n^{1/\gamma} \sim  \frac{n^{1/\gamma+1}}{t}\quad \text{with $n\gg t^{1/\gamma}$}.
$$
\end{remark}

\begin{figure}[ht]
\begin{center}
\begin{tabular}{p{0.45\textwidth}p{0.45\textwidth}}
\includegraphics[width=0.45\textwidth]{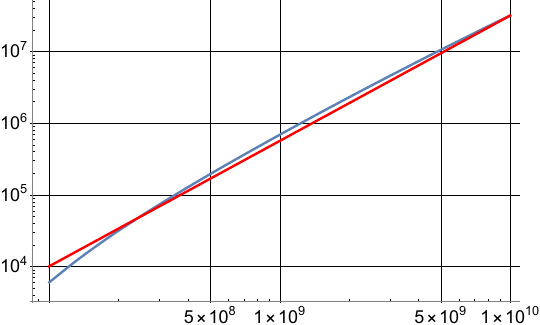} & \includegraphics[width=0.45\textwidth]{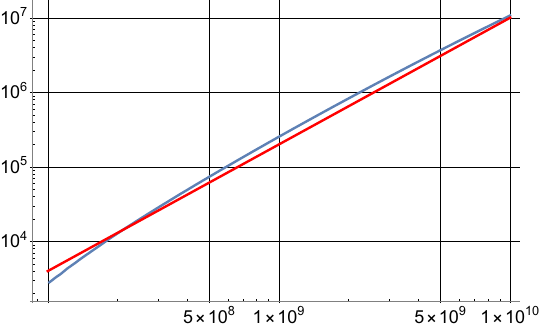} \\
\end{tabular}
\caption{\label{pic:3} The actual (blue line) and predicted (red line) number of endpoints of the second kind, which were obtained by time moment $n$ and remained non-cut up to time moment $t = 10^{10}$, with $n_0 = 500^2$ for a set of details $R_n$. The left graph corresponds to $\gamma = 4/3$, the right one does to $\gamma = 10/7$.}
\end{center}
\end{figure}

\begin{figure}[ht]
\begin{center}
\begin{tabular}{p{0.45\textwidth}p{0.45\textwidth}}
\includegraphics[width=0.45\textwidth]{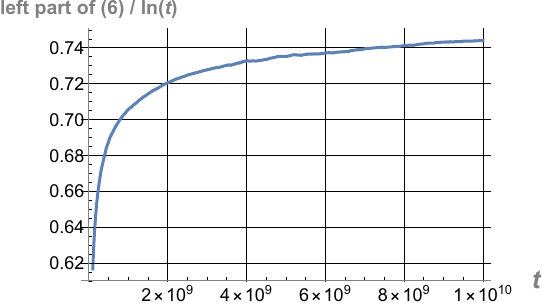} & \includegraphics[width=0.45\textwidth]{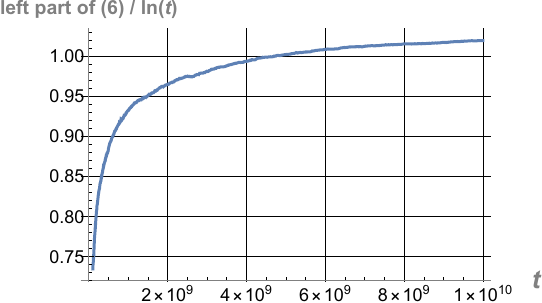} \\
\end{tabular}
\caption{\label{pic:4} The behavior of the mean value of the ratio of between the height and the width (see formula~\eqref{eq_prop}) divided by $\ln(t)$, for endpoints of the second kind with $\gamma=4/3$ and $\gamma=10/7$, $n_0=500^2$ for a set of details $R_n$. In the first case,  $\gamma^2-1$ equals $7/9\approx 0.778$, in the second one it does $51/49\approx 1.041$.}
\end{center}
\end{figure}

In Fig.~\ref{pic:3}, the last straight line (for $t=10^{10}$, $\gamma=4/3$ and $\gamma=10/7$) is colored in red. Real data on the number of endpoints of the second kind, which were obtained before time moment~$n$ and remained non-cut up to time moment~$t$, are shown in this figure as blue lines. Evidently, the red straight line well agrees with these data.

In Fig.~\ref{pic:4}, for the same values of $\gamma$, we adduce graphs that demonstrate the applicability of assertions of Lemma~\ref{appen2} to real data.

\end{appendices}

\end{document}